\documentclass[a4paper,11pt]{article}
\usepackage[utf8]{inputenc}
\usepackage{microtype}
\usepackage{mathtools, amsthm, amssymb, amscd, bezier, thmtools}

\usepackage{authblk}

\usepackage[dvips]{graphicx}
    \graphicspath{ {./images/} }
\usepackage{svg}
	\svgpath{{./images/}}
	\svgsetup{%
		inkscapepath=svgsubdir,
		inkscapearea=page
	}
\usepackage{epsfig}
\usepackage{psfrag}
\usepackage{wrapfig}
\usepackage{pstool}
\usepackage{overpic}

\usepackage[inline]{enumitem}
    \setlist{noitemsep,leftmargin=*}
	\setlist[1]{labelindent=\parindent}
	\setlist[itemize]{label=\raisebox{0.3ex}{\tiny\textbullet}}
	\setlist[enumerate]{label*=\arabic*.,labelsep=*,leftmargin=*}
	\setlist[enumerate,1]{ref=\arabic*}
	\setlist[enumerate,2]{ref=\theenumi.\arabic*}
	\setlist[enumerate,3]{ref=\theenumii.\arabic*}
	\setlist[enumerate,4]{ref=\theenumiii.\arabic*}

\usepackage{comment}

\usepackage{caption}

\usepackage[doi=false, isbn=false, url=false, eprint=false, maxbibnames=99, style=numeric, giveninits=true]{biblatex}
\usepackage[hidelinks]{hyperref}

\usepackage[noabbrev, capitalise]{cleveref} 
    \creflabelformat{equation}{#2#1#3}
    \crefname{section}{Section}{Sections}
    \crefname{subsection}{Subsection}{Subsections}
    \crefname{subsubsection}{Subsubsection}{Subsubsections}

\theoremstyle{plain}

\newtheorem{theorem}{Theorem}[section]

\newtheorem{proposition}[theorem]{Proposition}
\newtheorem{lemma}[theorem]{Lemma}

\theoremstyle{remark}

\newtheorem{question}{Question}

\theoremstyle{definition}
\newtheorem{definition}[theorem]{Definition}

\DeclareTextFontCommand{\defemph}{\bfseries}

\newcommand{\N}{\mathbb{N}}
\newcommand{\Z}{\mathbb{Z}}

\newcommand{\R}{\mathbb{R}}
\newcommand{\Sphere}{\mathbb{S}}

\newcommand{\set}[2]{\left\{{#1}\;\middle\vert\;{#2}\right\}}
\newcommand{\inv}{^{-1}}
\newcommand{\func}[5]{%
{#1}\colon {#2} &\longrightarrow {#3} \\%
{#4} &\longmapsto #5%
}

\DeclareMathOperator{\Int}{int}
\renewcommand{\setminus}{\smallsetminus}

\newcommand{\abs}[1]{\left| #1\right|}
\newcommand{\nor}[1]{\left\Vert #1\right\Vert}

\newcommand{\Switch}{\Sigma}
\newcommand{\Tange}{\Switch^{\mathrm{t}}}
\newcommand{\Cross}[1][]{\Switch^{\mathrm{c#1}}}
\newcommand{\Slide}[1][]{\Switch^{\mathrm{#1s}}}

\newcommand{\wt}{\widetilde}

\title{Characterization of non-deterministic chaos in two-dimensional non-smooth vector fields}

\author{%
Rodrigo D. Euz\'ebio\\
    \href{mailto:euzebio@ufg.br}{euzebio@ufg.br}\\
    Instituto de Matemática e Estatística, IME-UFG, Goi\^{a}nia-GO, Brazil.\and
Pedro G. Mattos\\
    \href{mailto:pedrogmattos@ime.unicamp.br}{pedrogmattos@ime.unicamp.br}\\
    Departamento de Matem\'atica, Estat\'istica e Computa\c c\~ao Cient\'ifica, IMECC-UNICAMP, Campinas-SP, Brazil.\and
R\'egis Var\~{a}o\\
    \href{mailto:varao@unicamp.br}{varao@unicamp.br}\\
    Departamento de Matem\'atica, Estat\'istica e Computa\c c\~ao Cient\'ifica, IMECC-UNICAMP, Campinas-SP, Brazil.
}

\begin{document}

\maketitle

\begin{abstract} 

Our context is Filippov systems defined on two-dimensional manifolds having a finite number of tangency points. We prove that topological transitivity is a necessary and sufficient condition for the occurrence of non-deterministic chaos when the Filippov system has non-empty sliding or escaping regions. A fundamental result for continuous flows is the equivalence of topological transitivity and existence of a dense orbit. We prove in our setting that topological transitivity for Filippov systems is indeed equivalent to the existence of a dense Filippov orbit, although, in contrast to the continuous case, we are not able to garantee that the dense orbit implies the existence of a residual set of dense orbits. Finally we prove that, in this context, topological transitivity implies strictly positive topological entropy for the Filippov system. This calculation is made using techniques similar to those from symbolic dynamics.
\end{abstract}



\section{Introduction}
\label{sec:introduction}

It is widely known that ordinary differential equations (ODEs) are one of the most relevant tools for modeling problems emerging from real-world applications. Nevertheless, in the last decades, the assumption that solutions of ODEs may experience a discontinuity has gained traction within the theory of dynamical systems. Some examples of discontinuous ODEs includes control theory, the dynamics of a bouncing ball, foraging predators, the anti-lock braking system (ABS), mechanical devices in which components collide with each other, problems with friction, sliding or squealing, etc. See the referred applications and others in \cite{diBern-relay}, \cite{Brogliato}, \cite{Rossa}, \cite{Dixon}, \cite{Genema},  \cite{Jac-To}, \cite{Kousaka}, \cite{Leine}, \cite{diBernardo-livro} as well as references therein.

Discontinuous ODEs can be treated under different hypotheses, one of the most known being due to Filippov (see \cite{Fi}) assuming that solutions can slide through a discontinuity region according to an \textit{a priori} defined rule. Accordingly, Filippov's solutions may experience non-uniqueness so non-typical behavior eventually emerges in relatively familiar scenarios such as planar phase portraits. Other approaches for dealing with discontinuous ODEs are nicely described in the text of \citeauthor{Cortes}~\cite{Cortes}. In this paper, we refer to ODEs with discontinuities by \textit{Filippov systems} since we assume Filippov's convention.

In this paper we address the problem of determining chaotic behavior of Filippov systems defined on two-dimensional manifolds. A well-accepted definition of chaos in the context of smooth deterministic dynamical systems is due to Devaney. A system which is topologically transitive, sensitive with respect to initial conditions and has density of periodic orbits is, in the classical dynamical system scenario, called chaotic or Devaney chaotic. A refinement of those conditions can be found in \cite{BBCDS}, where it is shown that topological transitivity and density of periodic solutions are sufficient for a system to be chaotic. An analogue of this definition is the one we use for Filippov systems. Since forward orbits of a Filippov system is non-deterministic, in this paper we are in fact considering {\it non-deterministic chaos}, more results on that the reader may see for instance \cite{BCE},  \cite{Colombo2011},  \cite{EJ} and \cite{NPV}. The definition of nondeterministic chaos is presented on Section \cref{sec:preliminaries}.

In this paper we assume that $M$ is a two-dimensional manifold with a Filippov system having a finite number of tangency points. We prove that the Filippov system is topologically transitive if, and only if, there is a dense Filippov orbit (\cref{theo:transitive-equivalence}). We prove on \cref{theo:complete} that a transitive Filippov system with non-empty sliding or escaping regions is non-deterministic chaotic (i.e. dense orbit, dense periodic orbits and sensitive with respect to initial conditions). In other words we are proving that \textbf{transitivity is equivalent to non-deterministic chaos}. But we can obtain even more: we also prove that, in our context, transitivity implies positive entropy (\cref{theo:entropy}). Frequently one uses entropy to determine some level of chaoticity of a given dynamics. That is because entropy is a number which may be understood as measuring the creation of new orbits---more precisely, the exponential growth rate of the number of different orbits of the system---hence a system with many ``genuinely'' distinct orbits (i.e. with positive entropy) might be a complex or ``chaotic'' system.

This paper is organized as follows: \cref{sec:main-results} states the main results of the paper. In \cref{sec:preliminaries}, we present the first concepts and definitions of Filippov systems which are going to be used throughout the paper. In \cref{sec:theo:complete}, we prove \cref{theo:complete}; in \cref{sec:transitive-equivalence}, \cref{theo:transitive-equivalence}; and \cref{theo:entropy} is done in \cref{sec:theo:entropy}.

\section{Main results}
\label{sec:main-results}

In the classical context of continuous flows, topological transitivity is equivalent to the existence of a dense orbit (which is sometimes simply called transitivity). In fact, it is further equivalence to the existence of a residual set of transitive points, that is, points through which a dense orbit passes. 
In order to provide the theory of Filippov systems with a solid foundation, we address topological transitivity for Filippov systems in this paper.

We obtain that the same equivalence of topological transitivity and transitivity holds for a specific setting on two-di\-men\-sional manifolds, but due to the complexity of piecewise smooth systems, it is not as straight-forward to obtain a residual set of transitive points.

\begin{theorem}
\label{theo:transitive-equivalence}
Assume that $M$ is a two-dimensional manifold with a Filippov system having a finite number of tangency points. Then the Filippov system is topologically transitive on $M$ if, and only if, there exists a dense Filippov orbit.
\end{theorem}

For continuous flows it is known that on a manifold topological transitivity is equivalent to the existence of a dense orbit and even more, there exists a residual set of dense orbit. One good question that we are not able to prove so far is:

\begin{question}
    In the context of \cref{theo:transitive-equivalence} does the existence of a dense orbit implies the existence of a residual set with dense orbit?
\end{question}

\begin{theorem}
\label{theo:complete}
Let $M$ be a two-dimensional manifold and $Z$ a transitive Filippov system having a finite number of tangency points. If the sliding region is non-empty, then the following statements hold:
	\begin{enumerate}
    \item There is a dense set $\Delta$ such that, for every $x \in \Delta$,
        \begin{enumerate}
        \item there is a dense orbit through $x$;
        \item the periodic orbits through $x$ form a dense set;
        \end{enumerate}
    \item The Filippov system has sensitive dependence on initial conditions.
	\end{enumerate} 
\end{theorem}

The above result we prove before proceeding to the proof of \cref{theo:transitive-equivalence} and, as questioned before, should we get the set $\Delta$ to be a residual set? After the proof of \cref{theo:complete} on \cref{ssec:comments_residual_set} we make a discussion concerning the issues to get a residual set. 

Our last result shows that topological transitivity (or transitivity) implies positive entropy, which may be seen as well as a measure of chaos. 

\begin{theorem}
\label{theo:entropy}
Assume that M is a two-dimensional manifold with a transitive Filippov system. If the sliding or escaping regions are non-empty, then the Filippov system has positive topological entropy.
\end{theorem}

\section{Preliminaries}
\label{sec:preliminaries}

\subsection{Filippov systems}
\label{ssec:preliminaries-filippov_systems}

Let $M$ be a 2-dimensional $C^k$ closed Riemannian manifold and $\Sigma\subset M$ a set formed by the union of $n$ smooth curves $\Sigma_i$ which are disjoint pairwise, where $\Sigma_i=h_i^{-1}(0)$ with $i=1,\ldots,n$ and $h_i:M\longrightarrow\mathbb{R}$ is a smooth function having $0$ as regular value. We write $\Sigma=\bigcup_{i=1}^n\Sigma_i$ and assume that $\Switch$ splits $M$ into $n+1$ disjoint regions $R_i$, on which we define $n+1$ vector fields $X_i$ ($i \in \{1, \ldots,n+1\}$). We call $\Switch$ the \defemph{switching manifold} and assume that it is contained in the boundary of the regions $R_i$. We denote by $d: M \times M \to \R$ the distance function induced by the Riemannian metric.

We call $\mathfrak{X}^r$ the space of C$^r$-vector fields $X:M\longrightarrow TM$ with $1\leq r \leq k$ where $k$ is sufficiently large. Call $\Omega$ the space of non-smooth vector fields $Z: M \rightarrow TM$ such that
    \begin{equation}
    \label{eq:Z}
    Z(x) =
        \begin{cases}
        X_i(x), &\text{if } x \in R_i \\
        \displaystyle\frac{Y_{i_2}h_i(x)Y_{i_1}-Y_{i_1}h_i(x)Y_{i_2}}{Y_{i_2}h_i(x)-Y_{i_1}h_i(x)}, &\text{if } x \in \Sigma_i
        \end{cases}
    \end{equation}
where $Y_{i_1}$ and $Y_{i_2}$ are some of the vector fields $X_i$ where $Y_{i_1}$ is the vector field associated to the positive part of $h_i$ and $Y_{i_2}$ associated to the negative part of $h_i$.

There are four kinds of points on $\Switch$. They can be classified as follows:
\begin{enumerate}
	\item At \defemph{crossing points}, both vector fields point to the same side of $\Switch$, so trajectories reaching such points immediately cross from one side to another.
	\item At (\defemph{unstable}) \defemph{sliding points}, the vectors fields point in opposite directions but towards $\Switch$ so that trajectories reach sliding in finite future time.
	\item At \defemph{escaping points} (or \defemph{stable sliding points}), the vectors fields point in opposite directions but now away from $\Switch$ so that trajectories reach escaping in finite past time.
	\item At \defemph{tangency points}, one of the vector fields is tangent to $\Switch$ and the other may be tangent or not. In this first case, we refer to it as a \defemph{regular tangency} point and, in the second case, a \defemph{double tangency} point.
\end{enumerate}

The above set of points are denoted, respectively, as $\Cross$, $\Slide[s]$, $\Slide[u]$ and $\Tange$. We also denote the (\defemph{general}) \defemph{sliding region} $\Slide := \Slide[u] \cup \Slide[s]$. Some of these points are represented in \cref{fig:orbitas_filippov}. In this work we also assume%
    \footnote{
        This is mainly for simplicity since we could only assume the pseudo-equilibrium points are separated and then consider in \cref{lemma:dense_set_reaching_sliding} only points of $\Slide$ that are not pseudo-equilibrium points.
    }
there are no \defemph{pseudo-equilibrium points}, that is, points $p \in \Slide$ where $Z(p) = 0$. In \cref{eq:Z} the field is defined on $\Switch$ using \textit{Filippov's convention}; it consists of a vector that is both tangent to $\Switch_i$ and is a convex combination of the vector fields $Y_{i_1}$ and $Y_{i_2}$ (see \cref{fig:convencao_filippov}).

\begin{figure}
    \centering
    \includesvg{orbitas_filippov}
    \caption{
        A Filippov system on $\Sphere^2$. The continuous path between the tangency points is a stable sliding region $\Slide[s]$ (notice the pseudo-equilibrium reached by the two dashed orbits).
    }
    \label{fig:orbitas_filippov}
\end{figure}

\begin{figure}
	\centering
		\includesvg{convencao_filippov}
		\caption{
            Defining a vector field on the switching manifold $\Switch$. The vector tangent to $\Switch$ is a convex combination of the other two vectors.
        }
	\label{fig:convencao_filippov}
\end{figure}

From \cref{eq:Z} we note that a non-smooth vector field $Z$ is also defined on points of sliding and escaping type through a convex combination of  $Y_{i1}$ and $Y_{i2}$. Those orbits sliding on $\Switch$ are formed by points on which trajectories collide to it for forward or backward finite time (see \cref{fig:orbitas_filippov}).

The following definition can be found in \cite{Marcel}.

\begin{definition}
A \defemph{Filippov orbit} (or \defemph{solution}) of $Z$ is a map $\gamma: \R \to M$ such that: if $\gamma(t)$ is outside $\Switch$, then the orbit is locally determined by the smooth vector fields $X_i$; once the orbit touches $\Switch$, then it is determined by the (sliding) vector field on $\Switch$ (\cref{eq:Z}), considering that, in the case the orbit goes through an escaping region, it may exit at any arbitrary moment.
\end{definition}

\begin{definition}
The \defemph{saturation} of a set $A \subset M$, denoted $A_\varphi$, is the union of every Filippov orbit with initial condition on $A$.
\end{definition}

For our purposes we will mainly deal with the set $\Slide_\varphi$.

\begin{definition}
\label{def:sets_that_reach_sliding}
We define the following sets:
\begin{enumerate}
    \item $E_+$ is the set of points of $M$ that first loose uniqueness going forward in $\Slide[s]$.
    \item $E_-$ is the set of points of $M$ that first loose uniqueness going backward in $\Slide[u]$.
    \item $D_+$ is the set of points of $M$ that have some Filippov orbit staring at which that reaches $\Slide[s]$ going forward.
    \item $D_-$ is the set of points of $M$ that have some Filippov orbit staring at which that reaches $\Slide[s]$ going backward.
\end{enumerate}
\end{definition}

Next, we introduce the definition of non-deterministic chaos. For simplicity, we will just say \textit{chaotic Filippov systems}. The definition is based on the classical definition of Devaney for chaos (see \cite{devaney}), but adapted for Filippov systems.

\begin{definition}
\label{definicao transitividade2}
Let $Z$ be a Filippov system on a manifold $M$ (as in \cref{eq:Z}). We define the following:
	\begin{enumerate}
		\item $Z$ is \defemph{topologically transitive} if given any two non-empty open sets $U$ and $V$ of $M$, there exists a Filippov orbit from a point of $U$ to a point of $V$.
    	\item $Z$ exhibits \defemph{sensitive dependence on initial conditions} if there exists a fixed $\delta > 0$ such that, for any non-empty open set $U$, there exist points $x, y \in U$, Filippov orbits $\gamma_x$ and $\gamma_y$ which start at $x$ and $y$, respectively, and some time $t$ such that $d(\gamma_x(t), \gamma_y(t)) > \delta$.
		\item a Filippov orbit $\gamma$ is \defemph{periodic} if there is $\tau \in \R$ such that $\gamma(t) = \gamma(t+\tau)$ for every $t \in \R$.
		\item $Z$ is \defemph{chaotic} if it is topologically transitive, has sensitive dependence on initial conditions and the union of all Filippov periodic orbits is a dense set.
	\end{enumerate}
\end{definition}

\subsection{Topology}

We briefly state some definitions and notation related to topology. We denote the interior of a set $A$ by $\Int(A)$ and the complement of a set $A$ by $A^\complement = M \setminus A$. A \defemph{residual} set is a set that is a countable intersection of sets whose interior is dense. This is equivalent to being the complement of a \defemph{meager} set, that is, a countable union of sets whose closure has empty interior. Residual sets represent a way to express the idea of ``almost all'' only using topology.

\subsection{Orbit spaces of Filippov systems}
\label{ssec:preliminaries-orbit_spaces}

In this subsection we follow definitions first presented in \cite{ACV} and later expanded on in \cite{GMV}. In the broad context of orbit spaces, we may take $M$ to be any bounded connected Riemannian manifold and the Filippov field $Z$ to be bounded (that is, its norm $\nor{Z}$, given by the supremum, must be bounded)\cite{GMV}. In our case here, we assume the particular context of the compact $2$-dimensional Riemannian manifold $M$ as exposed in the preceding \cref{ssec:preliminaries-filippov_systems}.

The \defemph{orbit space} of the Filippov system $Z$ over $M$ is the set of all possible (Filippov) orbits of $Z$, denoted $\wt M$. Denoting by $d: M \times M \to \R$ the distance function induced by the Riemannian of $M$, we define the \defemph{supremum distance} $\wt d_{\sup}$ on the orbit space $\wt M$ by setting, for each $\gamma_0, \gamma_1 \in \wt M$,
    \[
    \wt d_{\sup}(\gamma_0, \gamma_1) := \sum_{i \in \Z} \frac{1}{2^{\abs{i}}} \sup_{i \leq t < i+1} d(\gamma_0(t), \gamma_1(t)).
    \]
The distance converges for every pair of orbits because the space is compact. This turns $\wt M$ into a metric space that is complete, separable and has no isolated points \cite{GMV}.

The dynamics on $\wt M$ is induced by $Z$: we define a flow $\wt \Phi^t$ (for each $t \in \R$) as
    \begin{align*}
    \func{\wt\Phi^t}{\wt M}{\wt M}{\gamma}{
        \begin{aligned}[t]
        \func{\wt\Phi^t(\gamma)}{\R}{\R}{s}{\gamma(t+s)}.
        \end{aligned}
        }
    \end{align*}
It can be shown that the flow $\wt\Phi$ is continuous \cite{GMV}.

On the orbit space, we have uniqueness of solutions, which is most often not true for Filippov systems, so we may use concepts and techniques of the general theory of flows on $\wt M$ to try and obtain information about $Z$. One example of this is done in the next subsection for the definition of topological entropy for Filippov systems: we define the topological entropy of $Z$ as the topological entropy of the flow in the orbit space.

\subsection{Topological entropy}

To define the topological entropy of a Filippov system, we follow the approach of \cite{ACV}. We will first present the definition of topological entropy of Bowen-Dinaburg \cite{Bowen-EntropyGroupEndomorphismsHomogeneousSpaces} for compact metric spaces. We will assume our space is compact. We will also use generating sets and will not define separated sets, but a complete presentation can be found in \cite{viana-oliveira-ergodic-theory}. In this section we will denote the generic compact metric spaces by $K$ to avoid confusion with our previously defined manifold $M$.

\begin{definition}
Let $(K, d)$ be a compact metric space, $f: K \to K$ a continuous transformation, $\varepsilon>0$ and $n \in \N$. A \defemph{$(n, \varepsilon)$-generating set} is a set $E \subseteq K$ that satisfies the following: for every point $x \in K$, there is a point $a \in E$ such that $d(f^i(x), f^i(a)) < \varepsilon$ for every $i \in \{0, \ldots, n-1\}$.
\end{definition}

If we define the \defemph{dynamical ball} of center $a$, radius $\varepsilon$ and length $n$ (for the dynamics $f$) as the set
    \[
    B^n_f(a, \varepsilon) := \set{x \in K}{d(f^i(x), f^i(a)) < \varepsilon \text{ for } 0 \leq i < n},
    \]
the preceding definition is equivalent to the condition $K \subseteq \bigcup_{a \in E} B^n_f(a, \varepsilon)$.

We define $g^n(f,\varepsilon)$ to be the \defemph{smallest cardinality} of a $(n, \varepsilon)$-generating set, that is
    \[
    g^n(f, \varepsilon) := \min\set{\# E}{E \subseteq K \subseteq \textstyle\bigcup_{a \in E} B^n_f(a, \varepsilon)}.
    \]
By compactness, this number is always finite . 

The dynamical ball $B^n(a, \varepsilon)$ gives us information about which points stay $\varepsilon$-close to the center $a$ for $n$ units of time (i.e. $n$ iterations of $f$), so the minimal cardinality $g^n(f,\varepsilon)$ counts how many different orbits of the system there are, but for an approximated time $n$ and with an imprecision $\varepsilon$. We are interested in the exponential growth of the number of orbits of the system as time passes, so in the following definition we take the limit as $n \to \infty$ and $\varepsilon \to 0$ of the quantity $\frac{1}{n} \log g^n(f, \varepsilon)$. This limit can be shown \cite{viana-oliveira-ergodic-theory} to be a well-defined number in $\left[0, \infty \right]$.

\begin{definition}
Let $K$ be a compact metric space and $f: K \to K$ a continuous transformation. The \defemph{topological entropy of $f$} is
    \[
    h(f) := \lim_{\varepsilon \to 0} \limsup_{n \to \infty} \frac{1}{n} \log g^n(f, \varepsilon).
    \]
\end{definition}

We can define the \defemph{topological entropy a continuous flow} $\Phi^t$ ($t \in \R$) in a compact metric space $K$ in an analogous way using dynamical balls and generating sets, and its value is the same as the topological entropy of the time $1$ map $\Phi^1: K \to K$ of the flow \cite{viana-oliveira-ergodic-theory}, which may in fact be taken as the definition of entropy for flows. This is the motivation for the definition of topological entropy for Filippov systems, first presented in \cite{ACV}, which we will use.

\begin{definition}
\label{def:entropy.Filippov}
Let $Z$ a Filippov system on a manifold $M$. The \defemph{topological entropy} of $Z$ is the topological entropy of the time $1$ map $\wt \Phi^1: \wt M \to \wt M$ on the orbit space $\wt M$, denoted
    \[
    h(Z) := h(\wt \Phi^1).
    \]
\end{definition}

\section{Proof of \texorpdfstring{\cref{theo:complete}}{Theorem~\ref{theo:complete}}}
\label{sec:theo:complete}

We shall prove a number of important lemmas that will be needed for the proof of our result. 
%
%
First we show how topological transitivity leads to the existence of a connection of points on the sliding region.

\begin{figure} 
    \centering
    \includesvg{saturacao_sliding_escape}
    \caption{
        An open set $U$ that reaches $\Slide[s]$ flowing forward and an open set $V$ that reaches $\Slide[u]$ flowing backwards.
    }
    \label{fig:saturacao_sliding_escape}
\end{figure}

\begin{lemma}
\label{lemma:connecting_sliding}
Assume topological transitivity. For every pair of points $q_0$, $q_1 \in \Slide$, there is a Filippov orbit segment from $q_0$ to $q_1$.
\end{lemma}
\begin{proof}

The points $q_0$ and $q_1$ can be in either $\Slide[s]$ or $\Slide[u]$, so there are $4$ possibilities of connections of $q_0$ and $q_1$. We will describe how to connect them in all cases.
If $q_0 \in \Slide[s]$, since $\Slide$ is relatively open in $\Switch$ we can choose a point $q'_0$ before $q_0$ (that is, $q'_0$ is in the same connected component of $\Slide[s]$ as $q_0$ and it reaches $q_0$ through an orbit segment of the sliding vector field on $\Slide[s]$) and an open set $U_0 \subset M$ around $q'_0$ (that does not contain $q_0$) in such a way that every point of $U_0$ flows to $\Slide[s]$ and reaches $q_0$ (see \cref{fig:saturacao_sliding_escape,fig:connection_sliding}).

If $q_0 \in \Slide[u]$, we can choose an open set $U_0$ in an analogous way, but now inverting the direction of the flow of time: we take a point $q'_0$ after $q_0$ and an open set $U_0$ around it such that every point of $U_0$ must have have come from $\Slide[u]$ and, consequently, passed through $q_0$ (see \cref{fig:connection_sliding}).
In the same way we choose an open set $U_1 \subset M$ for $q_1$.

By topological transitivity, there is a Filippov orbit segment from a point $p_0 \in U_0$ to a point of $p_1 \in U_1$. In the case that $q_0 \in \Slide[s]$, by the choice of $U_0$ the orbit segment must have passed through $q_0$, so it can be restricted to an orbit segment from $q_0$ to $p_1$;
in the case that $q_0 \in \Slide[u]$, by the choice of $U_0$ the point $p_0$ must have flowed away from $q_0$, so the orbit segment can also be extended to one from $q_0$ to $p_1$.

On the other end, the situation is inverted. If $q_1 \in \Slide[s]$, the point $p_1$ must flow into $\Slide[s]$ and pass through $q_1$, so the orbit segment may be extended to one from $q_0$ to $q_1$;
if $q_1 \in \Slide[s]$, the point $p_1$ must have passed through $q_1$, so the orbit segment may also be extended to one from $q_0$ to $q_1$.
\end{proof}

\begin{figure}
    \centering
    \includesvg{connecting_sliding}
    \caption[Connection of two points on the sliding region.]{
        All the $4$ possibilities of connection of $2$ points $q_0$ and $q_1$ on the sliding region $\Slide$. The dashed orbit in the center of the figure represents that each orbit segment on the left can connect to each on the right.
    }
    \label{fig:connection_sliding}
\end{figure}


The next lemma establishes some properties of the sets described in \cref{def:sets_that_reach_sliding}.

\begin{lemma}
\label{lemma:dense_set_reaching_sliding}
Assume topological transitivity.
\begin{enumerate}
    \item Suppose $\Slide[s] \neq \emptyset$. Then $E_+$ is open and $D_+$ is dense.
    \item Suppose $\Slide[u] \neq \emptyset$. Then $E_-$ is open and $D_-$ is dense.
\end{enumerate}
\end{lemma}
\begin{proof}
We prove the first item, since the second is the same but with the direction of orbits inverted.

Let $p \in E_+$ and $q \in \Slide[s]$ be the point where $p$ first looses uniqueness going forward. Since $\Slide[s]$ is open in $\Switch$ and all points of the orbit from $p$ to $q$ are regular, we can find an open set around $p$ that first looses uniqueness in $\Slide[s]$. This shows $E_+$ is open.

Showing $D_+$ is dense is the same as showing $D_+^\complement$ has empty interior. For this, suppose for the sake of contradiction that there is a non-empty open set $U \subset D_+^\complement$. Since $\Slide[s] \neq \emptyset$, there is an open set $V$ such that all of its points reach $\Slide[s]$ (see \cref{fig:saturacao_sliding_escape}). By topological transitivity, there is an orbit from $U$ to $V$, and so this orbit can be extended to reach $\Slide[s]$, which contradicts the fact that $U$ is a set of points that do not reach $\Slide[s]$ by any orbit.
\end{proof}

In some examples of transitive Filippov systems in the literature like the bean model and the sphere model \cite{BCE,ACV,EJV}, it is easy to check that the sets $E_+$ and $E_-$ are dense, so in the following results we will assume this in order to obtain dense periodic orbits (for further discussion, check \cref{ssec:comments_residual_set}).


\begin{lemma}
\label{lemma:periodic_segment_open_set}
Assume topological transitivity and $\Slide \neq \emptyset$. Let $q_0 \in \Slide$ and $U \subseteq M$ be a non-empty open set. Then there is a periodic Filippov orbit segment through $q_0$ that intersects $U$.
\end{lemma}
\begin{proof}
We assume that $q_0 \in \Slide[s]$. The proof for $\Slide[u] \neq \emptyset$ is the same after inverting the direction of orbits. Since $q_0$ is in the sliding region, we may choose an open set $U_0 \subseteq M$ such that all of its points pass through $q_0$ going forwards, as done in \cref{lemma:connecting_sliding} (see \cref{fig:periodic_segment_open_set}). Since the set $E_+$ is open and we assume it is dense, take $p \in E_+ \cap U$ and an open neighborhood $V \subseteq E_+$ of $p$.

By topological transitivity, we choose a Filippov orbit segment $\gamma_0$ from $U_0$ to $U \cap V$. Since all the points of $U_0$ pass through $q_0$, $\gamma_0$ can be restricted to start at $q_0$, and since $\gamma_0$ passes through $V \subseteq E_+$, it can be extended to reach a point $q_1 \in \Slide[s]$. Now from \cref{lemma:connecting_sliding}, there is an orbit segment $\gamma_1$ from $q_1$ to $q_0$. Concatenating the orbit segments $\gamma_0$ and $\gamma_1$, we obtain a periodic orbit segment starting at $q_0$ that intersects $U$.
\end{proof}

\begin{figure}
    \centering
    \includesvg{periodic_segment_open_set}
    \caption[Periodic orbit connecting a point to an open set]{
        Given a point $q_0 \in \Slide$ and an open set $U$, we can find a periodic orbit segment staring at $q_0$ that intersects $U$. In the figure we depict the case in which $\Slide[s] \neq \emptyset$.
    }
    \label{fig:periodic_segment_open_set}
\end{figure}

Finally, we can prove the theorem.

\begin{proof}[Proof of \cref{theo:complete}]
We prove each item separately.
\begin{enumerate}
\item Take $x \in \Slide$.
    \begin{enumerate}
        \item Let $\{U_i\}_{i \in \N}$ be a countable base of $M$. By \cref{lemma:periodic_segment_open_set}, for each $U_i$ there is a periodic Filippov orbit segment $\gamma_i$ that starts at $x$ and intersects $U_i$. Define $\gamma$ to be the concatenation of the orbit segments $\gamma_i$ in order of index, backwards and forwards; i.e., define
            \[
            \gamma := \cdots \gamma_1 \cdot \gamma_0 \cdots \gamma_0 \cdot \gamma_1 \cdots.
            \]
        Now for every non-empty open set $U \subseteq M$, there is an open set $U_i$ contained in $U$, so $\gamma$ intersects $U$ because $\gamma_i$ intersects $U_i$. This shows $\gamma$ is dense in $M$.

        \item Now let $P$ be the union of (the image of) all the periodic orbits of $Z$ through $x$. By \cref{lemma:periodic_segment_open_set} it follows that, for every non-empty open set $U \subseteq M$, there is a periodic orbit segment $\gamma$ that passes through $x$ and intersects $U$, hence a periodic orbit through $x$ and $U$ which means that $P \cap U \neq \emptyset$, so $P$ is dense in $M$.
    \end{enumerate}
We now take $\Delta$ as the union of all the points on periodic orbits as constructed in the preceding items.

\item We will assume that $\Slide[s] \neq \emptyset$, but the proof for $\Slide[u] \neq \emptyset$ is analogous. First, we take $q \in \Slide[s]$ and two open sets $U_0$ and $U_1$, each one on each side of the connected component of $\Slide[s]$ on which $q$ is (see \cref{fig:sensitive_dependence_initial_conditions}), and not intersecting $\Slide[s]$. Now we choose two different periodic orbits $\beta_0$ and $\beta_1$ staring at $q$, similarly to what was done in \cref{lemma:periodic_segment_open_set}), but in this case we force each orbit to enter $\Slide[s]$ on the open sets $U_0$ and $U_1$, respectively, before they reach $q$. Denote their periods respectively by $b_0$ and $b_1$. Since $\Slide[s]$ is open, we may disturb these orbits slightly in order to have the ratio of their periods, $\frac{b_1}{b_0}$, be an irrational number. For each $i \in \{0,1\}$, we choose a point $p_i \in U_0$ through which orbit $\beta_i$ passes, and define a restricted orbit $\beta'_i$ from $q$ to $p_i$, and denote its period $b'_i$. Denote $\delta := d(U_0, U_1)$, the infimum of the distance between any point of $U_0$ and any point of $U_1$. Since each open set has been taken on one side of $\Slide[s]$, we have $\delta > 0$.

Now let $U$ be any non-empty open set. Since out set $\Delta$ as defined in the preceding item of this proof is dense, we may take $x, y \in \Delta \cap U$. Then there exist orbits $\alpha_x: [0,a_x] \to M$ and $\alpha_y: [0, a_y] \to M$ starting at $x$ and $y$, respectively, and ending at $q$. Now we define orbit $\gamma_x$ as the concatenation of $\alpha_x$ and $k_x \in \N$ orbits $\beta_0$, followed by the restricted orbit $\beta'_0$; that is, $\gamma_x := \alpha_x \cdot (\beta_0)^{k_x} \cdot \beta'_0$. Likewise, define $\gamma_y := \alpha_y \cdot (\beta_1)^{k_y} \cdot \beta'_1$, for $k_y \in \N$. We denote $c_x := a_x + k_x b_0 + b'_0$ and $c_y := a_y + k_y b_0 + b'_0$, which represent the time $x$ (or $y$) takes to traverse $\gamma_x$ (resp. $\gamma_y$) and reach $p_0$ (resp. $p_1$). Since the ratio $\frac{b_1}{b_0}$ is irrational, the integers $k_x$ and $k_y$ may be chosen in order that the $c_x$ and $c_y$ are as close as needed, such that (supposing $c_x \leq c_y$ without loss of generality) $\gamma_x(c_x) = p_0$ and $\gamma_y(c_y) \in U_1$, close to $p_1$. Defining $t := c_x$, this implies that $d(\gamma_x(t), \gamma_x(t) > \delta$.
\qedhere

\begin{figure}
    \centering
    \includesvg{sensitive_dependence_initial_conditions}
    \caption[Sensitive dependence on initial conditions.]{
        For any non-empty open set $U$, there are points $x$ and $y$ and orbits $\gamma_x$ and $\gamma_y$ such that, following these orbits, the two points eventually become $\delta$ apart. The orbits are constructed by connecting the points to a sliding region and then following periodic orbits for some time.
    }
    \label{fig:sensitive_dependence_initial_conditions}
\end{figure}

\end{enumerate}
\end{proof}

\subsection{Comments about finding a residual set}
\label{ssec:comments_residual_set}

In the classical case of smooth vector-fields, transitivity is equivalent to the existence of a residual set of points through which there is a dense orbit. In our setting, we have found a dense set $\Delta$ instead. As mentioned after the proof of \cref{lemma:dense_set_reaching_sliding}, for the bean model and the sphere model \cite{BCE, EJV} both sets $E_+$ and $E_-$ are dense, hence open dense sets. So we pose the following question:

\begin{question}
In the general context of a transitive Filippov system with a finite number of tangency points, are the sets $E_+$ and $E_-$ dense?
\end{question}

If both\footnote{Notice we are assuming here that both $\Slide[s]$ and $\Slide[u]$ are non-empty.} $E_+$ and $E_-$ were open dense sets, this would imply that their intersection $E := E_+ \cap E_-$ is a residual set (and hence a dense set by Baire category theorem). The proof of \cref{lemma:periodic_segment_open_set} can be easily adapted to be valid for points $x$ in this residual set $E$ (\cref{fig:periodic_segment_open_set_residual}). This can be done by starting with a point $x \in E$ instead of $q_0 \in \Slide$, and then connecting $x$ to $\Slide$ forward and backwards, which is possible by definition of $E$. After that, the final details of the proof would be almost the same as in \cref{lemma:periodic_segment_open_set}.

\begin{figure}
    \centering
    \includesvg{periodic_segment_open_set_residual}
    \caption[Adaptation of the proof of \cref{lemma:periodic_segment_open_set} for a residual set $\Delta$.]{
        Given a point $x \in \Delta$ and an open set $U$, we could find a periodic orbit segment staring at $x$ that intersects $U$.
    }
    \label{fig:periodic_segment_open_set_residual}
\end{figure}

\section{Proof of \texorpdfstring{\cref{theo:transitive-equivalence}}{Theorem~\ref{theo:transitive-equivalence}}}
\label{sec:transitive-equivalence}

\begin{proof}[Proof of \cref{theo:transitive-equivalence}]
	
If $\Slide \setminus \emptyset$, then the result follows from \cref{theo:complete}. The case $\Slide = \emptyset$ will follow from the two lemmas below. We now assume that the Filippov system is topologically transitive, since the converse is trivial.

\begin{lemma}\label{lemma:dense.T}
Assume topological transitivity, $\Slide = \emptyset$ and $\Sigma^{t}_\varphi$ is dense. Then there is a residual set of $M$ such that each point in this set has some dense Filippov orbit.
\end{lemma}
\begin{proof}
Now, recall that we are assuming a finite number of tangency points and our space is two dimensional. Hence, we may construct a Filippov orbit inside $\Sigma^{t}_\varphi$ which is dense.

Let us call this dense Filippov orbit as $\gamma(t)$ and let $t_0$ the time such that $\gamma(t_0)\in \Sigma^{t}$ and $\gamma(t)\cap \Sigma^{t}=\emptyset$ for all $t>t_0$. Therefore $\gamma([t_0,\infty])$ is dense in $M$ and it does not intersect any tangency point.

Let $g:M \rightarrow [0,1]$ be a smooth function such that $g^{-1}(0)=\Tange$. Now let us change the velocity of the Filippov orbit associated with the vector field of the system by multiplying it by the positive map $g$. Hence the new vector field $gZ$ is in fact a continuous vector field whose orbits coincide with the previous one, with the exception that now the tangency points are fixed ones. But now this new Filippov system is in fact a continuous flow, and for this new flow, $\gamma[t_0,\infty]$ is the image of some trajectory of it (which is dense).

Hence, in particular because there is a dense orbit for continuous flows we guarantee from classical results of continuous flows that there is a residual set $D$ of points with dense orbit. Notice that none of these dense orbits can pass through the singular points of $gZ$, hence these are also dense orbits for the Filippov system.

Notice that $\Sigma^{t}_\varphi$ is a meager set, hence $D \setminus \Sigma^{t}_\varphi$ is a residual set and the image of the orbit of these points for the flow $gZ$ is the same as for the original Filippov system.
\end{proof}

\begin{lemma}
\label{lemma:T.nondense}
Assume topological transitivity, $\Slide = \emptyset$ and $\Sigma^{t}_\varphi$ is not dense. Then there is a residual set of $M$ such that each point in this set has some dense orbit. Moreover, these dense orbits are regular orbits.
\end{lemma}
\begin{proof}
If $\Sigma^{t}_\varphi$ is not dense, then the interior of its complement $\left(\Sigma^{t}_\varphi\right)^\complement$ is non-empty. Let us take $\mathcal U \subset \Int[\left(\Sigma^{t}_\varphi\right)^\complement]=:A$. Now we prove that $A$ is dense. Indeed, given an open set $\mathcal V \subset M$, from the topological transitivity of the Filippov system there is a point $x \in \mathcal V$ and an orbit from $x$ to $\mathcal U$ which is regular since it outside the set where the break of uniqueness occurs. Hence there is an open set $\mathcal V_0 \subset V$ containing $x$ such that every segment of orbit from $\mathcal V_0$ to $\mathcal U$ is regular, which is true for $\mathcal V_0$ sufficiently small. Therefore from the invariance of the set $A$, we  get that $\mathcal V_0 \subset A$. That means the Filippov system on $A$ is topologically transitive, but on this invariant set $A$ the Filippov system determines a continuous flow since $A\cap\Sigma\subset\Cross$ because $\Slide = \emptyset$. Thus there is a residual set of points whose orbit is dense in $A$.
    
The lemma is proved, since $Z$ on $A$ is a continuous flows, $A$ is an open and dense set and $Z$ restricted to $A$ has a dense orbit.
\end{proof}

The above lemmas prove the result.
\end{proof}

\section{Proof of \texorpdfstring{\cref{theo:entropy}}{Theorem~\ref{theo:entropy}}}
\label{sec:theo:entropy}

We will now estimate the entropy of our topologically transitive Filippov system $Z$ to show it is strictly positive. We will use the orbit space $\wt M$ and its flow $\wt \Phi^t$, which is induced by $Z$. By \cref{def:entropy.Filippov}, the entropy of $Z$ is the entropy of the time $1$ map $\wt \Phi^1: \wt M \to \wt M$. We will define a subset $\wt \Gamma_p$ of $\wt M$ and consider on it the induced flow, and show that its entropy is strictly positive, which implies that the entropy of $\wt M$ is also positive since the entropy of a subsystem is always smaller than the entropy of the system \cite{viana-oliveira-ergodic-theory}.

To define the subset $\wt \Gamma_p$, we first need to have two different periodic orbits of the Filippov system $Z$ which have the same period and have a point in common. We start with a simple lemma that proves this is the case for our setting.

\begin{lemma}
\label{lemma:distinct_trajectories_with_same_period}
Assume topological transitivity and $\Slide \neq \emptyset$. There exist $2$ distinct periodic Filippov orbit segments $\gamma_0,\gamma_1: [0, \alpha] \to M$ of $Z$ with period $\alpha > 0$ and initial point $p = \gamma_0 (0) = \gamma_0(\alpha) = \gamma_1(0) = \gamma_1(\alpha) \in M$.
\end{lemma}
\begin{proof}
We can find two distinct periodic Filippov orbit segments $\eta_0$ and $\eta_1$ that have an initial point $p \in \Slide[s]$, in the same way it was done in \cref{lemma:connecting_sliding,lemma:periodic_segment_open_set}, and they can be forced to differ by making each one pass through each different side of the connected component of $\Slide$ being considered (see \cref{fig:saturacao_sliding_escape}). Since the orbit segments may have different periods, to obtain $\gamma_0$ and $\gamma_1$ with the same period we concatenate $\eta_0$ and $\eta_1$ in each possible order. The resulting orbit segments have the same period (the sum of the periods of $\gamma_0$ and $\gamma_1$).
\end{proof}

\subsection{Construction of the subsystem}

Now we present the construction of the subsystem $\wt \Gamma_p$ of the orbit space $\wt M$.
We take the periodic orbit segments $\gamma_0$ and $\gamma_1$ as in \cref{lemma:distinct_trajectories_with_same_period}. Let $\Gamma_0 := \gamma_0([0, \alpha])$ and $\Gamma_1 := \gamma_1([0, \alpha])$ denote their images, two curves on the base space $M$ that are distinct and have their initial point $p$ in common%
    \footnote{
    We do not need to assume the only intersection of the curves is $p$.
    }.
Denote the union of these curves as $\Gamma := \Gamma_0 \cup \Gamma_1$. The set $\wt \Gamma_p$ is the set of all the possible integral trajectories of $Z$ that start at $p$ and whose image lies on $\Gamma$.

All trajectories of $\wt \Gamma_p$ must start at $p$ and follow either $\Gamma_0$ or $\Gamma_1$, returning to $p$ after the period $\alpha$ has passed. Then, it must again follows either of the two curves and so on forwards and backwards in time. We can more formally describe this as follows. We define the set of indices $S := \{0, 1\}$ and let $S^\Z$ denote the space of all sequences $x: \Z \to S$, that is, sequences
    \[
    x = (\ldots, x_{-1}; x_0, x_1, \ldots)
    \]
such that $x_i \in S$. The shift map on $S^\Z$ is defined as the map $\sigma: S^\Z \to S^\Z$ such $\sigma(x) = (\ldots, x_0; x_1, x_2, \ldots)$; i.e., $\sigma(x)_i = x_{i+1}$ for each $i \in \Z$.
For each sequence $x \in S^\Z$, we can define an orbit $\gamma_x$ of $Z$ (which starts at $p$ and whose image lies on $\Gamma$) as the concatenation of the orbit segments $\gamma_0$ and $\gamma_1$ according to the entries of $x$: for each $i \in \Z$ and each $t \in \left[i\alpha, (i+1)\alpha \right[$, we have $\gamma_x(t) = \gamma_{x_i}(t-i\alpha)$. We will denote this by%
    \footnote{
    The use of curly braces is motivated by the similar notation used in the definition of functions by cases, for instance \cref{eq:Z}.
    }
    \[
    \gamma_x(t) =  \begin{cases}
                    \gamma_{x_i}(t-i\alpha), & i\alpha \leq t < (i+1)\alpha.
                    \end{cases}
    \]
This gives the characterization 
$\wt\Gamma_p = \set{\gamma_x}{x \in S^\Z}$.

To define the dynamics on $\wt \Gamma_p$, we take the map $\wt\Phi^\alpha: \wt\Gamma_p \to \wt\Gamma_p$ given on each $\gamma \in \wt\Gamma_p$ by $\wt\Phi^\alpha(\gamma)(t) = \gamma(\alpha+t)$. This is the flow map $\wt \Phi^\alpha: \wt M \to \wt M$ (as defined in \cref{ssec:preliminaries-orbit_spaces}) restricted to $\wt \Gamma_p$. The restriction is well defined because $\wt\Phi^\alpha(\gamma_x) = \gamma_{\sigma(x)} \in \wt\Gamma_p$ for every $x \in S^\Z$; this follows from the calculation
    \begin{align*}
    \wt\Phi^\alpha(\gamma_x)(t) &= \gamma_x(t+\alpha) \\
        &= \begin{cases}
            \gamma_{x_i}(t+\alpha-i\alpha), & i\alpha \leq t+\alpha < (i+1)\alpha
            \end{cases} \\
        &= \begin{cases}
            \gamma_{x_i}(t-(i-1)\alpha), & (i-1)\alpha \leq t < i\alpha
            \end{cases} \\
        &= \begin{cases}
            \gamma_{x_{i+1}}(t-i\alpha), & i\alpha \leq t < (i+1)\alpha
            \end{cases} \\
        &= \gamma_{\sigma(x)}(t).
    \end{align*} 
This shows that $\wt \Phi^\alpha: \wt\Gamma_p \to \wt \Gamma_p$ is a subsystem of $\wt \Phi^\alpha: \wt M \to \wt M$.

\subsection{Entropy calculation}

We are ready to calculate the entropy. Let us define the constant
    \[
    \mu := \sup_{0 \leq t < \alpha} d(\gamma_0(t),\gamma_1(t)),
    \]
which is the maximum distance points of $\gamma_0$ and $\gamma_1$ may be from each other at the same time. Since $\gamma_0$ and $\gamma_1$ are not equal for all times, $\mu>0$. This will be used in the estimates that follow.

Also, we define the quantity
    \[
    N(\gamma_x, \gamma_{x'}) := N(x,x') := \min\set{\abs{i}}{x_i \neq x'_i}.
    \]
between two orbits $\gamma_x$ and $\gamma_{x'}$. The quantity $N(x,x')$ is related to the distance function for the symbolic space $S^\Z$, and $N(\gamma_x, \gamma_{x'})$ serves a similar purpose in helping to estimate orbital distances, as the next lemma shows.

\begin{lemma}
\label{lem:bounded.N.implies.distant.orbits}
Let $m \in \N$ and $x, x' \in S^\Z$. If $N(\gamma_x, \gamma_{x'}) \leq m$, then
    \[
    d_{\sup}(\gamma_x, \gamma_{x'}) \geq \mu2^{-(m+1)\alpha}.
    \]
\end{lemma}
\begin{proof}
Denote $N := N(\gamma_x, \gamma_{x'})$. By definition of $N$, the orbits $\gamma_x$ and $\gamma_{x'}$ are different on the interval $\left[ N\alpha, (N+1)\alpha \right[$ or on the interval $\left[ -N\alpha, -(N-1)\alpha \right[$. Denote $u_i := \sup_{i \leq t < i+1} d(\gamma_x(t), \gamma_{x'}(t))$. Then $u_j = \mu$ for some $j \in \N$ such that $\lfloor N\alpha \rfloor \leq j < \lfloor (N+1)\alpha \rfloor$, so
    \[
    d_{\sup}(\gamma_x, \gamma_{x'}) = \sum_{i \in \Z} 2^{-\abs{i}} u_i \geq \sum_{i = \lfloor N\alpha \rfloor}^{\lfloor (N+1)\alpha \rfloor} 2^{-i} u_i \geq 2^{-j} \mu \geq 2^{-(m+1)\alpha}\mu.
    \qedhere
    \]
\end{proof}

Finally, we calculate the entropy of the subsystem $\wt \Phi^\alpha: \wt\Gamma_p \to \wt \Gamma_p$.

\begin{proposition}
\label{prop:topological.entropy.estimate.of.the.flow}
$h(\wt\Phi^\alpha|_{\wt\Gamma_p}) \geq \log(2)$.
\end{proposition}
\begin{proof}
To simplify notation, we will just write $\wt\Phi^\alpha$ instead of $\wt\Phi^\alpha|_{\wt\Gamma_p}$ inside this proof. We will first show that $g^n(\wt\Phi^\alpha, \mu2^{-(m+1)\alpha}) \geq 2^{2m+n}$. Let $E \subseteq \wt\Gamma_p$ be a set of orbits $\gamma_y$ such that $\# E \leq 2^{2m+n}-1$. Consider the set of $(2m+n)$-tuples
	\[
	F := \set{(y_{-m}, \ldots, y_{m+n-1}) \in S^{2m+n}}{\gamma_y \in E}.
	\]
Since there are at most $2^{2m+n}-1$ elements in $E$, then $\# F \leq 2^{2m+n}-1$; and since $\#(S^{2m+n}) = 2^{2m+n}$, there is at least one $(2m+n)$-tuple $(s_{-m}, \ldots, s_{m+n-1}) \in S^{2m+n} \setminus F$. Take a sequence $x \in S^\Z$ such that
	\[
	(x_{-m}, \ldots, x_{m+n-1}) = (s_{-m}, \ldots, s_{m+n-1}).
	\]
By the choice of the $s_i$, it follows that, for every $\gamma_y \in E$, $(y_{-m}, \ldots, y_{m+n-1}) \neq (x_{-m}, \ldots, x_{m+n-1})$. So there exist $l_y \in \{-m, \ldots,m\}$ and $k_y \in \{0,\ldots,n-1\}$ such that $y_{l_y+k_y} \neq x_{l_y+k_y}$, which is the same as $\sigma^{k_y}(y)_{l_y} \neq \sigma^{k_y}(x)_{l_y}$. Since $\abs{l_y} \leq m$, this implies that $N(\wt\Phi^{\alpha k_y}(\gamma_y), \wt\Phi^{\alpha k_y}(\gamma_x)) = N(\sigma^{k_y}(y),\sigma^{k_y}(x)) \leq m$, so it follows from \cref{lem:bounded.N.implies.distant.orbits} that
	\[
	{d_{\sup}}^n_{\wt\Phi^\alpha}(\gamma_y, \gamma_x) \geq d_{\sup}(\wt\Phi^{\alpha k_y}(\gamma_y), \wt\Phi^{\alpha k_y}(\gamma_x)) \geq \mu 2^{-(m+1)\alpha};
	\]
that is, $\gamma_x \notin B^n_{\wt\Phi^\alpha}(\gamma_y, \mu2^{-(m+1)\alpha})$. Since this is valid for every $\gamma_y \in E$, we conclude that $\gamma_x \notin \bigcup_{\gamma_y \in E} B^n_{\wt\Phi^\alpha}(\gamma_y, \mu2^{-(m+1)\alpha})$ and, because $E$ is an arbitrary set with $\# E \leq 2^{2m+n}-1$, it follows that $g^n(\wt\Phi^\alpha, \mu2^{-(m+1)\alpha}) > 2^{2m+n}-1$.

Finally, to estimate the entropy we note that, since $\mu2^{-(m+1)\alpha} \to 0$ as $m \to \infty$, and
    \[
    \lim_{m \to \infty} \limsup_{n \to \infty} \frac{1}{n} \log 2^{2m+n} = \lim_{m \to \infty} \limsup_{n \to \infty} \frac{2m+n}{n} \log(2) = \log(2),
    \]
so it follows from $g^n(\wt\Phi^\alpha, \mu2^{-(m+1)\alpha}) \geq 2^{2m+n}$ that $h(\wt\Phi^\alpha) \geq \log(2)$.
\end{proof}

We can now finish the proof of the main entropy theorem.

\begin{proof}[Proof of \cref{theo:entropy}]
From \cref{lemma:distinct_trajectories_with_same_period} we know that our Filippov system $Z$ has two periodic orbit segments $\gamma_0$ and $\gamma_1$, so the construction of the $\wt \Gamma_p$ can be done. From \cref{prop:topological.entropy.estimate.of.the.flow}, we have that $h(\wt\Phi^\alpha|_{\wt\Gamma_p}) \geq \log(2)$. Since $\wt \Phi^\alpha: \wt\Gamma_p \to \wt \Gamma_p$ is a subsystem of $\wt \Phi^\alpha: \wt M \to \wt M$, this means that $h(\wt\Phi^\alpha) \geq h(\wt\Phi^\alpha|_{\wt\Gamma_p})$. From the fact that $h(\wt\Phi^\alpha) = \alpha h(\wt\Phi^1)$ (the exponent property of entropy), it follows that
    \[
    h(Z) = h(\wt \Phi^1) = \alpha\inv h(\wt \Phi^\alpha) \geq \alpha\inv \log(2) > 0.
    \qedhere
    \]
\end{proof}


\section*{Acknowledgments}

The authors would like to thank Professor Marco A. Teixeira for useful conversations concerning this work.
R.E. was partially supported by Conselho Nacional de Desenvolvimento Cient\'ifico e Tecnol\'ogico (CNPq) (grants 402060/2022-9 and 308652/2022-3).
P.M. was partially financed by the Coordenação de Aperfeiçoamento de Pessoal de Nível Superior do Brasil (CAPES) (grant 141401/2020-6).
R.V. was partially supported by Conselho Nacional de Desenvolvimento Cient\'ifico e Tecnol\'ogico (CNPq) (grants 313947/2020-1 and 314978/2023-2), and partially supported by Fundação de Amparo à Pesquisa do Estado de São Paulo (FAPESP) (grants 17/06463-3 and 18/13481-0).

\sloppy
\printbibliography
 
\end{document}